\documentclass[reqno]{amsart}

\usepackage{hyperref}
\usepackage{amsmath,amsfonts,amsthm,amssymb,graphics, paralist}
\usepackage{eepic}

\newcommand{\g}{\geqslant}
\newcommand{\ar}{\rangle}
\newcommand{\al}{\langle}

\newcommand{\RR}{\mathbb{R}}

\newcommand{\CC}{\mathbb{C}}

\newcommand{\NN}{\mathbb{N}}
\newcommand{\p}{\partial}

\newcommand{\les}{\leqslant}
\newcommand{\lesa}{\lesssim}

\newcommand{\supp}{\text{supp }\,}

\theoremstyle{plain}
\newtheorem{theorem}{Theorem}
\newtheorem{proposition}[theorem]{Proposition}
\newtheorem{lemma}[theorem]{Lemma}
\newtheorem{corollary}[theorem]{Corollary}

\theoremstyle{remark}
\newtheorem{remark}{Remark}

\setdefaultenum{(i)}{(a)}{1}{A}
\title[Local and Global Well-posedness for CSD]{Local and Global well-posedness for the Chern-Simons-Dirac System in One Dimension}
\author[Nikolaos Bournaveas, Timothy Candy, and Shuji Machihara]{Nikolaos Bournaveas and Timothy Candy \vspace{0.3cm}\\
\textit{D\lowercase{epartment} \lowercase{of} M\lowercase{athematics}, U\lowercase{niversity} \lowercase{of} E\lowercase{dinburgh}\\
E\lowercase{dinburgh} EH9 3JE, U\lowercase{nited} K\lowercase{ingdom} \vspace{0.6cm}\\}
Shuji Machihara \vspace{0.3cm}\\
\textit{D\lowercase{epartment}  \lowercase{of} M\lowercase{athematics}, F\lowercase{aculty} \lowercase{of} E\lowercase{ducation}, S\lowercase{aitama} U\lowercase{niversity},\\
 255 S\lowercase{himo}-O\lowercase{kubo}, S\lowercase{akura}-\lowercase{ku}, S\lowercase{aitama} C\lowercase{ity} 338-8570, J\lowercase{apan} }}


\begin{document}
\date{\today}

\maketitle
\begin{abstract}
We consider the Cauchy problem for the Chern-Simons-Dirac system on $\RR^{1+1}$ with initial data in $H^s$. Almost optimal local well-posedness is obtained. Moreover, we show that the solution is global in time, provided that initial data for the spinor component has finite charge, or $L^2$ norm.
\end{abstract}

\section{Introduction}
The Chern-Simons action was first studied from a geometric point of view in \cite{Chern1974}. Subsequently, it was proposed as an alternative gauge field theory to the standard Maxwell theory of electrodynamics on Minkowski space $\RR^{1+2}$ \cite{Deser1982}. As well as being of interest theoretically, it has also been successfully applied to explain phenomena in the physics of planar condensed matter, such as the fractional quantum Hall effect \cite{Lopez1991}. Recently, much progress has been made on the Cauchy problem for the Chern-Simons action coupled with various other field theories such as Chern-Simons-Higgs, \cite{Bournaveas2009a, Huh2007}, and Chern-Simons-Dirac \cite{Huh2007}. \\

In the current article we consider the Cauchy problem for the Chern-Simons-Dirac (CSD) system in $\RR^{1+1}$. This system was
first studied by Huh in \cite{Huh2010} as a simplified version of the more standard CSD system on $\RR^{1+2}$. The CSD system on $\RR^{1+1}$ is given by
\begin{equation}\label{CSD} \tag{CSD}
    \begin{split}
            i \gamma^\mu D_\mu \psi &= m \psi\\
            \p_t A_1 - \p_x A_0 &= \psi^\dagger \alpha \psi\\
            \p_t A_0 - \p_x A_1 &= 0
    \end{split}
\end{equation}
with initial data $\psi(0)= f$, $A(0) = a$, where the spinor $\psi$ is a $\CC^2$ valued function of $(t, x)=(x_0, x_1) \in \RR^{1+1}$ and the gauge components $A_0$ and $A_1$ of the gauge $A= (A_0, A_1)$
are real valued. The covariant derivative is given by $D_\mu = \p_\mu - i A_\mu$ and we raise and lower indices with respect to the metric $g = \text{diag} (1, -1)$. Repeated indices are summed over $\mu =0, 1$, and we use $\psi^\dagger$ to denote the conjugate transpose of $\psi$.
We take the standard representation of the Gamma matrices
            $$\gamma^0 = \begin{pmatrix} 1 & 0 \\ 0 & -1 \end{pmatrix}, \qquad \gamma^1 = \begin{pmatrix} 0 & 1 \\ -1 & 0 \end{pmatrix}$$
and let $\alpha= \gamma^0$.

The system (\ref{CSD}) is interesting from a mathematical point of view for a number of reasons. Firstly solutions to (\ref{CSD}) satisfy conservation of charge, i.e. we have $\| \psi(t) \|_{L^2} = \| f\|_{L^2}$ for any $t \in \RR$. This is similar to the Dirac-Klein-Gordon (DKG) equation where conservation of charge also holds. We remark that conservation of charge forms a crucial component in the study of global existence for DKG \cite{Selberg2007, Tesfahun2009}. On the other hand, conservation of charge fails for other quadratic Dirac equations which have been studied in the literature   \cite{Bournaveas2008,  Machihara2005, Machihara2007}. Secondly, there is substantial null structure in the nonlinear terms in (\ref{CSD}), in the sense that (\ref{CSD}) is roughly equivalent to a system of nonlinear  wave equations of the form
                $$ \Box \Psi = Q( \Psi, \Psi)$$
where $Q(\Psi, \Psi)$ is a combination of the null forms $Q_{ij} = \p_i \Psi_\mu \p_j \Psi_\nu - \p_j \Psi_\mu \p_i \Psi_\nu$ and $Q_0 = g^{\mu \nu} \p_\mu \Psi \p_\nu \Psi$.
Moreover the structure of the equation means that in the mass free case $m=0$, the spinor $\psi$ can be explicitly solved in terms of the initial data $\psi_0$ and the gauge $A$.  This idea was used in \cite{Huh2010} to derive a number of interesting observations on the asymptotic behaviour of solutions to (\ref{CSD}) as $t \rightarrow \infty$.\\

Currently the best known results for the Cauchy problem for (\ref{CSD}) are due to Huh in \cite{Huh2010} where it was shown that the (\ref{CSD}) system is locally well-posed for initial data in the charge class $(\psi_0, a_0)\in L^2 \times L^2$, and globally well-posed for $(\psi_0, a_0) \in H^1\times H^1$. To prove the local in time result, Huh rewrote (\ref{CSD}) as a system of nonlinear wave equations and showed that the nonlinear terms contained null structure. The null form estimates of Klainerman and Machedon \cite{Klainerman1993} then completed the proof.

In the current article we use a different approach. Instead of rewriting (\ref{CSD}) as a wave equation, we factor the Dirac and Gauge components into null-coordinates $x\pm t$ and use Sobolev spaces adapted to these coordinates.  In one space dimension, Sobolev spaces based on null coordinates seem to behave better than the closely related $X^{s, b}_\pm$ type spaces of Bourgain-Klainerman-Machedon which have been used in many other low-regularity results on Dirac equations in one dimension, see for instance the results in \cite{Candy2010, Machihara2010}.  Our main result is the following.

\begin{theorem}\label{thm - local existence}
   Let $\frac{-1}{2} < r \les s \les r+1$ and $(f, a) \in H^{s} \times H^{r}$. Then there exists $T>0$ and a solution $ (\psi, A) \in C\big( [-T, T], H^{s}\times H^r\big)$ to (\ref{CSD}). Moreover solution depends continuously on the initial data,  is unique in some subspace of $C\big( [-T, T], H^{s} \times H^r \big)$, and any additional regularity persists in time\footnote{ More precisely, if $(\psi_0, a_0) \in H^{s'} \times H^{r'}$ with $s' \g s$, $r' \g r$, and $r'\les s' \les r'+1$,  then we can conclude that $(\psi, A) \in C\big( [-T, T], H^{s'} \times H^{r'} \big)$, where $T$ only depends on the size of
              $ \| f \|_{H^{s}} + \| a\|_{H^{r}} $.}.
\end{theorem}

\begin{remark} If we set $m=0$, then solutions to (\ref{CSD}) are invariant under the scaling $(u, A) \mapsto \frac{1}{\lambda}(u, A)\big(\frac{t}{\lambda}, \frac{x}{\lambda} \big)$. Hence the scale invariant space is $\dot{H}^{-\frac{1}{2}}\times \dot{H}^{-\frac{1}{2}}$. Since we do not expect any well-posedness below the scaling regularity, the range of well-posedness in Theorem \ref{thm - local existence} is essentially optimal, except possibly at the endpoint $r=\frac{-1}{2}$. Moreover, it should be possible to show that (\ref{CSD}) is ill-posed in some sense outside of the range given in Theorem \ref{thm - local existence} by using the techniques in \cite{Machihara2010}, but we do not consider the problem of ill-posedness here.
\end{remark}

The local existence portion of Theorem \ref{thm - local existence} will follow by the standard iteration argument, using estimates contained in \cite{Machihara2010}. The proof of uniqueness is  more difficult and does not follow directly from the existence proof, primarily because the spaces used to prove existence do not scale nicely on the domain $[-T, T]\times \RR$. Instead we will need to prove a more precise version of an energy inequality from \cite{Machihara2010}. See Proposition \ref{prop - energy for Z with time dependence} below. Finally the persistence of regularity is quite interesting as it allows both the regularity of the spinor, $\psi$, and the gauge, $A$, to be varied independently, provided that we remain in the region of well-posedness.\\

We now turn to the question of global well-posedness. In the case $s\g0$ we can exploit the conservation of charge together with a decomposition argument from \cite{Candy2010} to obtain the following.

\begin{corollary}\label{cor - global existence}
  Assume that $s \g0$ in Theorem \ref{thm - local existence}. Then the local solution can be extended to a global solution $ (\psi, A) \in C\big( \RR, H^{s}\times H^r \big)$.
\end{corollary}

\begin{figure}
\setlength{\unitlength}{0.254mm}
\begin{picture}(382,236)(285,-245)
        \allinethickness{0.254mm}\path(320,-165)(320,-170) % Plain Solid Line
        \allinethickness{0.254mm}\path(320,-160)(320,-165) % Plain Solid Line
        \allinethickness{0.254mm}\path(485,-165)(485,-170) % Plain Solid Line
        \allinethickness{0.254mm}\path(485,-165)(485,-160) % Plain Solid Line
        \put(305,-186){\shortstack{$-\frac{1}{2}$}} % Plain Text
        \put(485,-186){\shortstack{$1$}} % Plain Text
        \put(650,-181){\shortstack{$s$}} % Plain Text
        \put(360,-21){\shortstack{$r$}} % Plain Text
        \put(485,-41){\shortstack{$s=r$}} % Plain Text
        \put(595,-41){\shortstack{$s=r+1$}} % Plain Text
        \allinethickness{0mm}\special{sh 0.1}\path(375,-165)(375,-220)(430,-220)(605,-45)(495,-45)(375,-165) % Color Polygon
        \allinethickness{0.254mm}\path(605,-45)(430,-220) % Plain Solid Line
        \allinethickness{0.254mm}\path(495,-45)(320,-220) % Plain Solid Line
        \allinethickness{0.254mm}\dottedline{5}(320,-220)(430,-220) % Plain Dotted Line
        \allinethickness{0.254mm}\path(485,-165)(485,-160) % Plain Solid Line
        \allinethickness{0.254mm}\put(375,-245){\vector(0,1){235}} % Vector Line
        \allinethickness{0.254mm}\put(285,-165){\vector(1,0){375}} % Vector Line
         % Set color to black again (default font color)
\end{picture}

\caption{The domain of local/global well-posedness from Theorem \ref{thm - local existence} and Corollary \ref{cor - global existence}. We have local existence inside the lines $s=r$ and $s=r+1$ for $r>-\frac{1}{2}$. Global existence holds inside the shaded region. }
\end{figure}
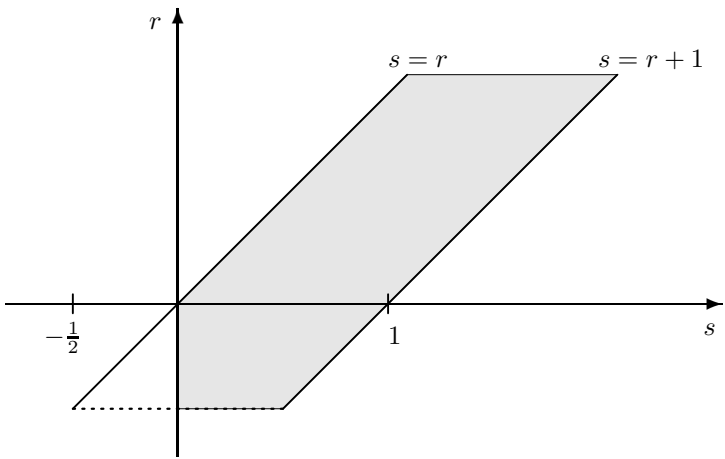

We now give a brief outline of this article. In Section \ref{sec estimates} we gather together the estimates we require in the proof of Theorem \ref{thm - local existence}. The local existence component of Theorem \ref{thm - local existence} is proven in Section \ref{sec local existence}. The proof of uniqueness is contained in Section \ref{sec uniqueness}. Finally in Section \ref{sec global existence} we prove Corollary \ref{cor - global existence}.

\subsection*{Notation}
Throughout this paper $C$ denotes a positive constant which can vary from line to line. The notation $a \lesa b $ denotes the inequality $a \les C b$.
 We let $L^p(\RR^n)$ denote the usual Lebesgue space.  Occasionally we write $L^p(\RR^n) = L^p$ when we can do so without causing confusion. This comment also applies to the other function spaces which appear throughout this paper. If $X$ is a metric space and $I\subset \RR $ is an interval,  then  $C(I, X)$ denotes the set of continuous functions from $I$ into $X$. For $s \in \RR$, we define $H^s$ to be the usual Sobolev space defined using the norm
            $$ \| f \|_{H^s (\RR)} = \| \al \xi \ar^s \widehat{f} \|_{L^2(\RR)}$$
 where  $\widehat{f}$ denotes the Fourier transform of $f$ and  $\al \xi \ar = (1+ |\xi|^2)^{\frac{1}{2}}$. The space-time Fourier transform of a function $\psi(t, x)$ is denoted by $\widetilde{\psi}(\tau, \xi)$. We also use the notation $\mathcal{F}_y(f)$ to denote the Fourier transform of $f$ with respect to the variable $y$.

 If $X$ is a Banach space of functions defined on $\RR^n$, then for an open set $\Omega \subset \RR^n$ we define the restriction space $X(\Omega)$ by restricting elements of $X$ to $\Omega$. If we equip $X(\Omega)$ with the norm
            $$ \| f \|_{X(\Omega)} = \inf_{g=f \text{ on } \Omega} \| g\|_{X}$$
 then $X(\Omega)$ is also a Banach space. Finally, for $a, b, c \in \RR$ we use the notation $c \prec \{ a, b\}$ to denote that either
            $$ a+b \g 0, \qquad c\les \min\{a, b\}, \qquad c< a+b -\frac{1}{2}$$
 or
             $$ a+b > 0, \qquad c< \min\{a, b\}, \qquad c\les a+b -\frac{1}{2}$$
 holds. Note that $c \prec \{ a, b\}$ implies that the following product inequality for Sobolev spaces holds
            $$\| f g\|_{H^c(\RR)} \lesa \| f\|_{H^a(\RR)} \| g\|_{H^b(\RR)}.$$

\section{ Estimates}\label{sec estimates}

The main estimates we require in the proof of Theorem \ref{thm - local existence} have already been proven in \cite{Machihara2010}. Define
            $$ \| u \|_{ Z^{s, b}_\pm} = \big\| \al \tau \mp \xi \ar^s \al \tau \pm \xi \ar^b \widetilde{\psi}(\tau, \xi) \big\|_{L^2_{\tau , \xi}}.$$
Note that $Z_\pm^{s, b} $ is just the product Sobolev space in the null directions $x\pm t$. The  $Z^{s, b}_\pm$ space is enough to control the nonlinear terms in (\ref{CSD}). However for $s$ close to $\frac{-1}{2}$, the space $Z_\pm^{s, b}$ is not contained inside $C(\RR, H^s(\RR))$. Thus to prove the local well-posedness result in Theorem \ref{thm - local existence}, we will need to add a component to control the $L^\infty_t H^s$ norm. To this end, following \cite{Machihara2010}, we define the space $Y^{s, b}_\pm$ by using the norm
            $$ \| u \|_{Y^{s, b}_{\pm} } = \big\| \al \xi \ar^s \al \tau \pm \xi \ar^b \widetilde{u}(\tau, \xi) \big\|_{L^2_\xi L^1_\tau}.$$
It is easy to see that
            $$ \| u \|_{L^\infty_t H^s_x} \les \| u \|_{Y^{s, 0}_{\pm}}.$$
and so $Z^{s, b}_\pm \bigcap Y^{s,0}_\pm \subset C(\RR, H^s(\RR)$. We remark that spaces of the form $Y^{s, b}_\pm$ have been used previously to augment the standard $X^{s, b}$ spaces for $b=\frac{1}{2}$ in the periodic case in \cite{Bourgain1993}, see also \cite{Ginibre1997}.

The first result we will need is the following energy type inequality.

\begin{lemma}[\cite{Machihara2010} Lemma 3.2] \label{lem - energy est}
Let $s, b \in \RR$ and $S=[-1, 1]\times \RR$. Suppose $u$ is a solution to
            \begin{align*} \p_t u \pm \p_x u &= F \\
                                          u(0) &= f \end{align*}
on $S$. Then
        \begin{equation}\label{lem - energy est - eqn main}
                \| u \|_{Z^{s, b}_\pm(S)} + \| u \|_{Y^{s, 0}_\pm(S)} \lesa \| f\|_{H^s} + \inf_{F'|_S = F }\Big( \| F' \|_{Z^{s, b-1}_\pm} + \| F' \|_{Y^{s, -1}_\pm} \Big)
        \end{equation}
where the infimum is over all  $F' \in Z^{s, b-1}_\pm \cap Y^{s, -1}_\pm $ with $F' = F$ on $S$.

\end{lemma}

The previous energy inequality is sufficient to prove existence of solutions to (\ref{CSD}), however to obtain uniqueness we will require a slightly more refined version of Lemma \ref{lem - energy est} which we leave to Section \ref{sec uniqueness}.

To close the iteration argument we will need the following nonlinear estimate contained in \cite{Machihara2010}.

\begin{lemma}[\cite{Machihara2010}, Lemma 3.4]\label{lem - Y nonlinear est}
Let $s_1, s_2, b_1, b_2, s  \in \RR$ and assume there exists $a_0, b_0 \in \RR$ such that
       \begin{equation} \begin{split}  a_0 \prec \{s_1, b_2\}, \qquad b_0 \prec \{ s_2, b_1\},  \qquad s \prec\{a_0 , b_0 +1 \} \\
                            s_1 + b_1 > \frac{-1}{2}, \qquad s_2 + b_2 >  \frac{-1}{2}. \end{split} \end{equation}
Then we have
        $$ \| uv \|_{Y^{s, -1}_\pm} \lesa \| u\|_{Z^{s_1, b_1}_\pm} \| v \|_{Z^{s_2, b_2}_\mp}.$$
\end{lemma}

We also have the following well known product estimates for Sobolev spaces.

\begin{lemma}\label{lem - Z product estimate}
  Assume $s \prec \{s_1, b_2\}$ and $b \prec \{b_1, s_2\}$. Then
            $$ \| u v \|_{Z^{s, b}_{\pm} } \lesa \| u \|_{Z^{s_1, b_1}_{\pm}} \| v\|_{Z^{s_2, b_2}_\mp }.$$
\end{lemma}
Finally we will need the following Lemma which will help simplify the arguments leading to uniqueness.

\begin{lemma}\label{lem - time dependence of Sobolev norms}
Let $\frac{-1}{2}<s<\frac{1}{2}$ and $0<T<1$. Assume $\rho \in H^1$ and let $\rho_T(t) = \rho\big( \frac{t}{T} \big)$. Then
    \begin{equation}\label{lem - time dependence of Sobolev norms - main eqn 1} \| \rho_T(t) f(t) \|_{H^s_t} \lesa_{\rho} \| f\|_{H^s} \end{equation}
with constant independent of $T$. Consequently
        \begin{equation}\label{lem - time dependence of Sobolev norms - main eqn 2} \| \rho_T(t) u \|_{Z^{s, 0}_\pm} \lesa_\rho \| u \|_{Z^{s, 0}_\pm} \end{equation}
with constant independent of $T$.
\begin{proof}
The inequality (\ref{lem - time dependence of Sobolev norms - main eqn 1}) is well-known. For the convenience of the reader we
sketch the proof in the  appendix.  To prove (\ref{lem - time dependence of Sobolev norms - main eqn 2})  we use a change of variables
        \begin{align*} \| \rho_T(t) \psi \|_{Z^{s, 0}_\pm} &= \Big\| \al \tau \mp \xi \ar^s \int \widehat{\rho_T}(\lambda) \widetilde{\psi}(\tau- \lambda, \xi) d\lambda \Big\|_{L^2_{\tau, \xi}} \\
        &= \Big\| \al \tau  \ar^s \int \widehat{\rho_T}(\lambda) \widetilde{\psi}(\tau \pm \xi- \lambda, \xi) d\lambda \Big\|_{L^2_{\tau, \xi}}
    \end{align*}
and then apply (\ref{lem - time dependence of Sobolev norms - main eqn 1}).
\end{proof}
\end{lemma}

\section{Local Existence}\label{sec local existence}

We start by noting that if we let $u_\pm = \psi_1 \pm \psi_2$ and $A_\pm = A_0 \mp A_1$, we can rewrite (\ref{CSD}) in the form
    \begin{equation} \label{Dirac pm} \begin{split}
        i ( \p_t u_+ + \p_x u_+ ) &= m u_- - A_- u_+  \\
          i ( \p_t u_- -  \p_x u_- ) &= m u_+ - A_+ u_- \\
                                    u_\pm(0) &= f_\pm
    \end{split} \end{equation}
and
    \begin{equation}\label{Gauge pm} \begin{split}
        \p_t A_+ + \p_x A_+ &= -\Re( u_+ \overline{u}_-) \\
        \p_t A_- - \p_x A_- &= \Re( u_+ \overline{u}_-) \\
                        A_\pm(0) &= a_\pm \end{split}
    \end{equation}
where $f_\pm = f_1 \pm f_2$, $a_\pm = a_0 \mp a_1$, and we use $\Re(z)$ to denote the real part of $z \in \CC$. The formulation (\ref{Dirac pm}), (\ref{Gauge pm}) is much easier to work with than (\ref{CSD}) as the null structure is more apparent. Namely all the nonlinear terms involve products of the form $\psi_+ \phi_-$ which behave far better than the  product $\psi_+ \phi_+$, see for instance the estimates in \cite{Selberg2010b}. The fact that the nonlinear terms in (\ref{Dirac pm}) and (\ref{Gauge pm}) are all + - products is a reflection of the null structure present in the (\ref{CSD}) system.
\newline

We will deduce Theorem \ref{thm - local existence} from the following.

\begin{theorem}\label{thm - small data LWP}
  Let $\frac{-1}{2}< r \les s \les r+1$ and assume $f \in H^s$, $a \in H^r$. Choose $r^*>\frac{-1}{2}$ with $s-1 \les r^* \les r$. Then there exists $\epsilon>0$ such that if $|m|<\epsilon$ and
            $$ \| f \|_{H^r} + \|a \|_{H^{r^*}} < \epsilon$$
  then there exists a solution $(\psi, A) \in C\big( [-1, 1], H^s\times H^r\big)$ to (\ref{CSD}) with $(\psi, A)(0) = (f, a)$. Moreover solution depends continuously on the initial data and if we let $u_\pm = \psi_1 \pm \psi_2$ and $A_\pm = A_0 \mp A_1$ then
            $$ u_\pm \in Z^{s, b}_\pm(S) \cap Y^{s, -1}_\pm(S), \qquad A_\pm \in Z^{r, b}_\pm(S) \cap Y^{s, -1}_\pm(S)$$
  for any $b>\frac{1}{2}$ with $s\les b\les r^*+1$ and $S=[-1, 1] \times \RR$.
\end{theorem}
\begin{figure}
\setlength{\unitlength}{0.254mm}
\begin{picture}(382,236)(285,-245)
        \allinethickness{0.254mm}\put(375,-245){\vector(0,1){235}} % Vector Line
        \allinethickness{0.254mm}\put(285,-165){\vector(1,0){375}} % Vector Line
        \allinethickness{0.254mm}\path(375,-165)(495,-45) % Plain Solid Line
        \allinethickness{0.254mm}\path(375,-165)(320,-220) % Plain Solid Line
        \allinethickness{0.254mm}\dottedline{5}(320,-220)(430,-220) % Plain Dotted Line
        \allinethickness{0.254mm}\path(430,-220)(595,-55) % Plain Solid Line
        \allinethickness{0.254mm}\path(595,-55)(605,-45) % Plain Solid Line
        \allinethickness{0.254mm}\path(320,-165)(320,-170) % Plain Solid Line
        \allinethickness{0.254mm}\path(320,-160)(320,-165) % Plain Solid Line
        \allinethickness{0.254mm}\path(485,-165)(485,-170) % Plain Solid Line
        \allinethickness{0.254mm}\path(485,-165)(485,-160) % Plain Solid Line
        \put(305,-186){\shortstack{$-\frac{1}{2}$}} % Plain Text
        \put(485,-186){\shortstack{$1$}} % Plain Text
        \put(650,-181){\shortstack{$s$}} % Plain Text
        \put(360,-21){\shortstack{$r$}} % Plain Text
        \put(485,-41){\shortstack{$s=r$}} % Plain Text
        \put(595,-41){\shortstack{$s=r+1$}} % Plain Text
        \allinethickness{0.254mm}\dottedline{5}(515,-100)(515,-135) % Plain Dotted Line
        \allinethickness{0.254mm}\dottedline{5}(515,-100)(440,-100) % Plain Dotted Line
        \allinethickness{0.254mm}\dottedline{5}(515,-135)(440,-135) % Plain Dotted Line
        \allinethickness{0.254mm}\dottedline{5}(440,-135)(440,-100) % Plain Dotted Line
        \put(518,-95){\shortstack{$(s, r)$}} % Plain Text
        \put(425,-150){\shortstack{$(r, s-1)$}} % Plain Text
        \put(512,-103){\shortstack{$\bullet$}}
        \put(437,-138){\shortstack{$\bullet$}}
        %\put(515,-100){\special{color rgb 0 0 0}{\ellipse*{3}{3}}} % Color Dot
        %\put(440,-135){\special{color rgb 0 0 0}{\ellipse*{3}{3}}} % Color Dot
         % Set color to black again (default font color)
\end{picture}
\caption{The time of existence given by the rescaled version of Theorem \ref{thm - small data LWP} at regularity $H^s\times H^r$, only depends on the size of the initial data at the regularity $H^r \times H^{s-1}$ (provided $s-1>\frac{-1}{2}$).}
\label{figure - time of existence}
\end{figure}
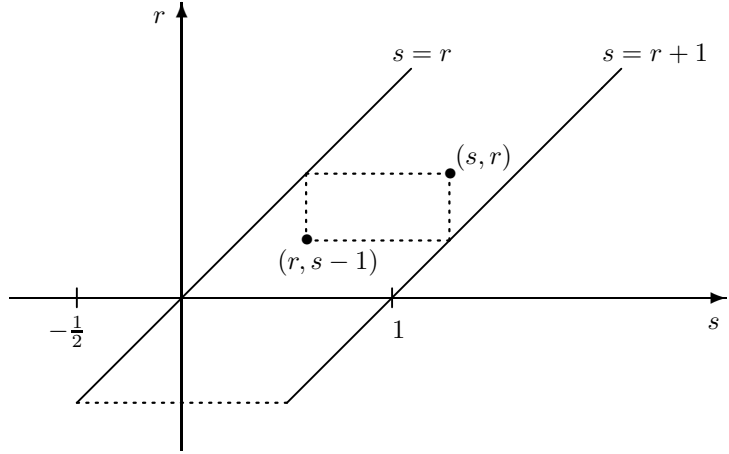

Assume for the moment that Theorem \ref{thm - small data LWP} holds, we deduce Theorem \ref{thm - local existence} as follows. Let $(f, a) \in H^{s}\times H^{r}$ with $\frac{-1}{2}< r\les s\les r+1$. Theorem \ref{thm - small data LWP} together with a scaling argument then gives a solution $(\psi, A) \in C([-T, T], H^s \times H^r)$ that depends continuously on the initial data, where $T$ only depends on some negative power of $\| f\|_{H^r} + \| a\|_{H^{r^*}}$ with $r^*>\frac{-1}{2}$ and $s-1\les r^* \les r$, see Figure \ref{figure - time of existence}. The uniqueness we leave till the next section. Hence to complete the proof of Theorem \ref{thm - local existence} it only remains to check that any additional regularity persists in time.

Suppose the initial data has additional smoothness $(f, a) \in H^{r^*} \times H^{s^*}$ with $s^*>s$, $r^*>r$, and $r^*\les s^* \les r^* + 1$. Applying the local existence result we have $(\psi , A) \in C\big( (-T^*, T^*), H^{s^*} \times H^{r^*}\big)$ for some $T^*>0$. Persistence of regularity will follow if we can obtain $T^* \g T$. To this end, we note that it is enough to show that if $T^*< T$ and
               \begin{equation}\label{persistance argument}  \limsup_{t \rightarrow T^*} \big(\| \psi(t) \|_{H^{s^*}} + \| A(t) \|_{H^{r^*}} \big)= \infty \end{equation}
then we also have
                \begin{equation}\label{persistence argument end} \limsup_{t \rightarrow T^*} \big( \| \psi(t) \|_{H^{s}} + \| A(t) \|_{H^{r}}\big) = \infty. \end{equation}
This is done in steps as follows. We first deduce by the rescaled version of Theorem \ref{thm - small data LWP} that
               \begin{equation}\label{persistence argument eqn2} \limsup_{t\rightarrow T^*} \big( \| \psi(t) \|_{H^{r^*}} + \| A(t) \|_{H^{\max\{s^*-1, \frac{-1}{2} + \epsilon\}}} \big)= \infty\end{equation}
for any sufficiently small $\epsilon>0$. Since if not, then we can choose some sequence of points $t_n \rightarrow T^*$ with $\sup_{n} \| \psi(t_n)\|_{H^{r^*}} + \| A(t_n) \|_{H^{\max\{ s^*-1, \frac{-1}{2} + \epsilon\}}}<\infty$. Taking $t_n$ sufficiently close to $T$ and applying a rescaled version of Theorem \ref{thm - small data LWP} with initial data $(\psi(t_n), A(t_n))$, we can extend our solution beyond $T^*$, contradicting (\ref{persistance argument}). Thus provided $T^*<\infty$ and (\ref{persistance argument}) holds, we must have (\ref{persistence argument eqn2}).

Repeating this argument again with (\ref{persistance argument}) replaced with (\ref{persistence argument eqn2}) we obtain
                $$ \limsup_{t\rightarrow T^*} \big(\| \psi(t) \|_{H^{\max\{s^*-1, \frac{-1}{2} + \epsilon\}}} + \| A(t) \|_{H^{\max\{ r^*-1, \frac{-1}{2} + \epsilon\}}} \big)= \infty.$$
We now continue in this manner and observe that after $k$ iterations, the $H^{\max\{s^*-k, \frac{-1}{2} + \epsilon\}} \times H^{\max\{r^*-k, \frac{-1}{2} + \epsilon\}}$ norm must blowup as we approach $T^*$. Taking $k$ such that $s^*-k \les s$ and $r^*-k\les r$ we obtain (\ref{persistence argument end}) as required.
\newline

We now come to the proof of small data local well-posedness for (\ref{CSD}).

  \begin{proof}[Proof of Theorem \ref{thm - small data LWP}]
  Let $\frac{-1}{2}<r \les s \les r+1$ and choose $b > \frac{1}{2}$ with $s \les b \les r^*+1$. Note that this is possible since $r^* \g s-1$ and $r^*>\frac{-1}{2}$. Let $r\les s' \les s$. We claim that Lemma \ref{lem - Y nonlinear est} and Lemma \ref{lem - Z product estimate} imply the estimates
        \begin{equation}\label{small data nonlinear est1}
            \| u v \|_{Y^{r, -1}_{\pm} } \les \| u v \|_{Y^{s', -1}_\pm} \lesa \| u \|_{Z^{s', b}_\pm} \| v\|_{Z^{r^*, b}_\mp}
        \end{equation}
  and
        \begin{equation}\label{small data nonlinear est2}
            \| u v \|_{Z^{r, b-1}_\pm} \les \| u v\|_{Z^{s', b-1}_\pm } \lesa \| u \|_{Z^{s', b}_\pm} \| v\|_{Z^{r^*, b}_\mp}.
        \end{equation}
  To obtain the estimate (\ref{small data nonlinear est1}), an application of Lemma \ref{lem - Y nonlinear est} reduces the problem to showing that there exists $a_0, b_0 \in \RR$ such that
     \begin{align*}  a_0 \prec \{s', b\}, \qquad b_0 \prec \{ r^* , b\},  \qquad s' \prec\{a_0 , b_0 +1 \} \\
                            s' + b > \frac{-1}{2}, \qquad r^* + b >  \frac{-1}{2}. \end{align*}
     Since $r^*\les r \les s' \les s \les b$, we let $a_0 = s'$, $b_0 = r^*$. It is clear that $s' \prec \{ s', b\}$ and $r^* \prec\{ r^*, b\}$. Thus the only remaining conditions are
            $$s'+ r^*+ 1 \g 0, \qquad \qquad s' \les r^*+1, \qquad s'< s'+r^*+1 - \frac{1}{2}.$$
     But these also hold provided $r^*, s' > \frac{-1}{2}$ and $s' \les r^*+1$, which follows since $s' \les s\les r^*+1$. Consequently (\ref{small data nonlinear est1}) holds.

     The remaining estimate, (\ref{small data nonlinear est2}), follows from  Lemma \ref{lem - Z product estimate} provided that
            $$ s' \prec\{ s', b\}, \qquad \qquad b-1 \prec \{ r^*, b\}.$$
     Using the assumptions $s', r^* > \frac{-1}{2}$ and $ b> \frac{1}{2}$ this reduces to
            \begin{align*}
              s'\les b &,  &\,\,\,\,s'<s'+b - \frac{1}{2} \\
              b-1\les r^* &, &b-1< r^*+b - \frac{1}{2}.
            \end{align*}
     These inequalities also hold in view of the assumptions $\frac{-1}{2} <s' \les b$ and $\frac{1}{2} <  b \les r^*+1$. Therefore (\ref{small data nonlinear est1}) and (\ref{small data nonlinear est2}) both hold.

It suffices to consider the system (\ref{Dirac pm}) and (\ref{Gauge pm}) with the assumption
        $$ \sum_{\pm} \| f_\pm \|_{H^r} + \| a_\pm \|_{H^{r^*}} < \epsilon.$$
Let $S = [-1, 1] \times \RR$ and define the Banach space $E^s = \big\{ v = (v_+, v_-) \, \big| \, v_\pm \in Z^{s, b}_\pm(S) \cap Y^{s, 0}_\pm(S)\big\}$ with norm
        $$ \| v \|_{E^s} = \sum_{\pm} \| v_\pm \|_{Y^{s, 0}_{\pm} (S)} + \|v_\pm \|_{Z^{s, b}_\pm(S)} $$
Note that since $Y^{s, 0}_\pm (S) \subset L^\infty_t H^s_x(S)$ we have $\| v\|_{L^\infty_t H^s_x(S)} \lesa \| v\|_{E^s}$. Let $\Gamma = \sum_{\pm} \| f_\pm \|_{H^s} + \| a_\pm \|_{H^r}$ and define the closed subset $\mathcal{X}_\epsilon \subset E^s \times E^r$ by
        $$\mathcal{X}_\epsilon = \{  \| u \|_{E^r} + \| A \|_{E^{r^*}} \les 2 C \epsilon \}\cap \{ \| u \|_{E^s} + \|A\|_{E^r} \les 2C \Gamma\}. $$
Define the  map $\mathcal{S} : \mathcal{X}_\epsilon \longrightarrow \mathcal{X}_\epsilon$ by letting $\mathcal{S}(u, A) = (v, B)$ be the solution to
        \begin{equation}
            \begin{split}
                i ( \p_t \pm \p_x) v_\pm &= m u_\mp  + A_\mp u_\pm \\
                (\p_t \pm \p_x) B_\pm &= \pm \Re( u_+ \overline{u}_-) \\
                                    v_\pm(0) &= f_\pm, \qquad B_\pm(0) = a_\pm.
            \end{split}
        \end{equation}
Then using Lemma \ref{lem - energy est} together with (\ref{small data nonlinear est1}) and (\ref{small data nonlinear est2}) we obtain
    $$ \| v\|_{E^s} + \| B\|_{E^r} \lesa \sum_\pm\big( \| f_\pm \|_{H^s} + \| a_\pm \|_{H^r}\big) + |m| \big( \| u\|_{E^s} + \| A\|_{E^r}\big) + \big( \| u\|_{E^r} + \| A\|_{E^{r^*}}\big) \big(\| u\|_{E^s} + \| A\|_{E^r}\big) $$
and
     $$ \| v\|_{E^r} + \| B\|_{E^{r^*}} \lesa \sum_\pm \big(\| f_\pm \|_{H^r} + \| a_\pm \|_{H^{r^*}}\big) + |m| \big( \| u\|_{E^r} + \| A\|_{E^{r^*}}\big) + \big( \| u\|_{E^r} + \| A\|_{E^{r^*}}\big)^2.$$
The assumption $(u, A) \in \mathcal{X}_\epsilon$ then gives the inequalities
    \begin{align*}
         \| v\|_{E^s} + \| B\|_{E^r} &\les C \Gamma +   C \epsilon \Gamma + C^2 \epsilon \Gamma\\
         \| v\|_{E^r} + \| B\|_{E^{r^*}} &\les C \epsilon + C \epsilon^2 + C^2 \epsilon^2
    \end{align*}
Therefore, provided $\epsilon$ is sufficiently small, depending only on the constants in (\ref{small data nonlinear est1}), (\ref{small data nonlinear est2}), and (\ref{lem - energy est - eqn main}), we see that $\mathcal{S}$ is well defined. A similar argument shows that $\mathcal{S}$ is a contraction mapping, consequently we have existence, uniqueness in $\mathcal{X}_\epsilon$, and continuous dependence on the initial data.
\end{proof}

\section{Uniqueness}\label{sec uniqueness}

In this section we will complete the proof of Theorem \ref{thm - local existence} and show that the solution obtained in Section \ref{sec local existence} is unique. More precisely, we will prove the following.
\begin{proposition}\label{prop - uniqueness}
Let $\frac{-1}{2}< r \les s \les r+1$,  $T>0$, and $b>\frac{1}{2}$. Define $S_T=[-T, T]\times \RR$. Assume $(u, A)$ and $(v, B)$ are solutions to (\ref{Dirac pm}) and (\ref{Gauge pm}) with $u_\pm, v_\pm \in Z_\pm^{s, b}(S_T)$ and $A_\pm, B_\pm \in Z_\pm^{r, b}(S_T)$. If $(u, A)(0) = (v, B)(0)$ then $(u, A) = (v, B)$ on $S_T.$
\end{proposition}

The proof of Proposition \ref{prop - uniqueness} is slightly involved as we need to understand the behaviour of the energy inequality Lemma \ref{lem - energy est} on the domain $S_T$ for small $T$. For the $Y^{s, b}$ component this is reasonably straightforward.

\begin{lemma}\label{lem - Y energy est}
Let $s \in \RR$, $0<T<1$, and $0<\epsilon<1$. Suppose $\psi$ is a solution to
            \begin{align*} \p_t \psi \pm \p_x \psi &= F \\
                                          \psi(0) &= f. \end{align*}
Let $\rho \in C^\infty_0$ and define $\rho_T(t) = \rho\big( \frac{t}{T} \big)$. Then
    \begin{align}\label{lem - Y energy est - main eqn} \|  \rho_T(t) \psi \|_{Y^{s, 0 }_\pm} &\lesa_\rho \|f\|_{H^s} + \big\| \al \xi \ar^s \min\{ T, |\tau \pm \xi|^{-1} \} \widetilde{F} \big\|_{L^2_\xi L^1_\tau}\\
    &\lesa \| f\|_{H^s} + T^\epsilon \| F\|_{Y^{s, \epsilon -1}_\pm} \label{lem - Y energy est - main eqn 2}\end{align}
with constant independent of $T$.
\begin{proof}
It is easy to see that (\ref{lem - Y energy est - main eqn}) follows from  the estimate
     \begin{equation}\label{lem - Y energy est - mod energy inhom term}
                \Big\| \mathcal{F}_t\Big[ \rho_T(t) \int_0^t e^{\mp i(t-s)\xi} \widehat{F}(s) ds\Big](\tau, \xi) \Big\|_{L^2_\xi L^1_\tau} \lesa \big\| \min\{ T, |\tau \pm \xi|^{-1} \} \widetilde{F} \big\|_{L^2_\xi L^1_\tau }. \end{equation}
 Note that by scaling it is sufficient to consider the case $T=1$. Consequently $\min\{ 1, |\tau \pm \xi |^{-1} \} \approx \al \tau \pm \xi \ar^{-1}$ and so (\ref{lem - Y energy est - mod energy inhom term}) follows from  Lemma 3.2 in \cite{Machihara2010}. The remaining inequality (\ref{lem - Y energy est - main eqn 2}) then follows by observing that since $0<T<1$,
            $$ \min\{ T, |\tau \pm \xi|^{-1} \} \lesa T^\epsilon \al \tau \pm \xi \ar^{\epsilon -1}.$$
\end{proof}
\end{lemma}

It remains to control $Z^{s, b}_\pm$ component of the energy inequality. This is significantly more difficult as both multipliers $\al \tau + \xi \ar$ and $\al \tau - \xi \ar$ involve the time variable. This observation, together with the fact that $Y^{s, 0}$ has a different scaling to $Z^{s, b}$, is the main difficulty in the following proposition.

\begin{proposition}\label{prop - energy for Z with time dependence}
Let $\frac{-1}{2} < s \les 0$ and $0<T<1$. Choose $0<\epsilon<\frac{1}{2}$ and let $\frac{1}{2}<b<\min\{1+s, 1-\epsilon\}$. Assume $\rho, \sigma \in C^\infty_0$ with $\rho(t)=1$ for  $t \in [-1, 1]$, $\sigma(t) = 1$ for $t \in \supp \rho$, and
    $$\supp \rho  \subset \supp \sigma \subset [-2, 2].$$
Define  $\rho_T(t) = \rho\big( \frac{t}{T} \big)$ and $\sigma_T(t) = \sigma\big( \frac{t}{T}\big)$. Let $\psi$ be a solution to
            $$ \p_t \psi \pm \p_x \psi = F.$$
Then
        \begin{equation}\label{prop - energy for Z with time dependence - main eqn} \| \rho_T(t) \psi \|_{Z^{s, b}_\pm} \lesa T^{\frac{1}{2} - b} \| \sigma_T(t) \psi \|_{Y^{s, 0}_\pm}  + T^\epsilon \| F \|_{Z^{s, b-1 +\epsilon}_\pm}\end{equation}
with the implied constant independent of $T$.
\begin{proof}
  We only prove the $+$ case as the $-$ case is similar. Note that since $\sigma_T(t)=1$ on $\supp \rho_T$ we may simply write $\psi = \sigma_T \psi$. Let $\Omega \subset \RR^2$ and define
        $$I(\Omega) = \Big\| \al \tau + \xi \ar^b \al \tau - \xi \ar^s \int_\RR \widehat{\rho_T}(\tau - \lambda) \widetilde{\psi}(\lambda, \xi) d\lambda \Big\|_{L^2_{\tau, \xi}(\Omega)}.$$
  We break $\RR^2$ into different regions and estimate each region separately. We first consider the set
        $$\Omega_1 = \{|\tau + \xi | \les T^{-1}\}$$
  and split this into the regions $2|\tau - \xi| \g |\xi|$ and $2|\tau - \xi | \les |\xi|$. In the former region, since $s\les 0$ and $\al \tau + \xi \ar^b \les T^{-b}$,
    \begin{align*}
        I(\Omega_1 \cap \{2|\tau - \xi | \g |\xi|\} )  &\lesa T^{-b} \Big\| \int_\RR \widehat{\rho_T}(\tau - \lambda) \al \xi \ar^s \widetilde{\psi}(\lambda, \xi) d\lambda \Big\|_{L^2_{\tau, \xi}} \\
        &\lesa T^{-b} \| \widehat{\rho_T}(\tau)\|_{L^2_\tau} \| \psi \|_{Y^{s, 0}_+} \\
        &\lesa_\rho T^{\frac{1}{2} - b} \| \psi \|_{Y^{s, 0}_+}.
    \end{align*}
On the other hand if $2|\tau - \xi| \les |\xi|$ then $|\tau| \approx |\xi| \approx |\tau + \xi| \lesa T^{-1}$. Hence
    \begin{align*}
        I(\Omega_1 \cap \{2|\tau - \xi| \les |\xi|\}) &\lesa \Big\| \al \tau + \xi \ar^{b-s} \al \tau - \xi \ar^s \int \widehat{\rho_T}(\tau - \lambda) \al \xi \ar^s \widehat{\psi}(\lambda, \xi) d\lambda \Big\|_{L^2_{\tau, \xi}(|\tau - \xi|, |\tau + \xi| \lesa T^{-1} ) } \\
          &\lesa T^{s-b} \| \widehat{\rho_T} \|_{L^\infty_\rho} \| \al \tau \ar^s \|_{L^2_\tau( |\tau| \les T^{-1}) } \| \psi \|_{Y^{s, 0}_+} \\
        &\lesa_\rho T^{s - b} \times T \times T^{ - \frac{1}{2} - s} \,\| \psi \|_{Y^{s, 0}_+}  \\
        &= T^{\frac{1}{2}-b} \| \psi \|_{Y^{s, 0}_+}.
    \end{align*}
Therefore
        $$ I( \Omega_1 )  \lesa T^{\frac{1}{2} - b} \| \psi \|_{Y^{s, 0}_+} .$$

We now consider the region $\Omega_2 = \{|\tau + \xi| \g T^{-1}\}$. Note that
        \begin{align*}
            \big(\rho_T(t) \psi\big)^{\widetilde{\,\,\,\,\,\,} }(\tau, \xi) &= \frac{1}{ i( \tau + \xi) } \int i \big((\tau - \lambda) + (\lambda + \xi) \big)  \widehat{\rho_T}(\tau - \lambda) \widetilde{\psi}(\lambda, \xi)  d\lambda\\
                            &= \frac{1}{i(\tau + \xi)} \Big[ T^{-1}  \big( (\p_t\rho)_T \psi\big)^{\widetilde{\,\,\,\,\,\,}}(\tau, \xi) + \big(\rho_T F\big)^{\widetilde{\,\,\,\,\,\,}}(\tau, \xi)\Big]
        \end{align*}
and so, using the fact that $|\tau + \xi| \g T^{-1} \gg 1$ implies $|\tau + \xi | \approx \al \tau + \xi \ar$, we have
    \begin{align}\label{prop - energy for Z with time dependence - decomp into terms}
        I(\Omega_2) \les T^{-1} \Big\|  \al \tau + \xi \ar^{b-1} \al \tau - \xi \ar^s \big( (\p_t\rho)_T \psi\big)^{\widetilde{\,\,\,\,\,\,}} \Big\|_{L^2_{\tau, \xi}(\Omega_2)} + \Big\|  \al \tau + \xi \ar^{b-1} \al \tau - \xi \ar^s \big( \rho_T F\big)^{\widetilde{\,\,\,\,\,\,}} \Big\|_{L^2_{\tau, \xi}(\Omega_2)}. \end{align}
We estimate each of these terms separately. For the first term we follow the $\Omega_1$ case and decompose $\Omega_2$ into $2|\tau - \xi| \g |\xi|$
and $2|\tau - \xi | \les |\xi|$. In the former region we use the fact that $\al \tau + \xi \ar^{b-1} \les T^{1-b}$ to deduce that
\begin{align*}
    T^{-1} \Big\| \al \tau + \xi \ar^{b-1} \al \tau - \xi \ar^s \int \widehat{(\p_t\rho)_T}(\tau - \lambda) &\widetilde{\psi}(\lambda, \xi) d\lambda \Big\|_{L^2_{\tau, \xi}(\Omega_2 \cap \{ 2|\tau - \xi| \g |\xi|\})} \\
    &\lesa T^{-b} \Big\| \int \widehat{(\p_t\rho)_T}(\tau - \lambda) \al \xi \ar^s \widetilde{\psi}(\lambda , \xi) d\lambda \Big\|_{L^2_{\tau , \xi}} \\
    &\lesa T^{-b} \| \widehat{(\p_t\rho)_T}\|_{L^2} \| \psi \|_{Y^{s, 0}_+} \\
    &\lesa_\rho  T^{\frac{1}{2} - b} \| \psi \|_{Y^{s, 0}_+}.
\end{align*}
On the other hand for $2|\tau - \xi |\les |\xi|$ we have $|\tau + \xi | \approx |\xi|$ and so
    \begin{align*}
            T^{-1} \Big\| \al \tau + \xi \ar^{b-1} \al \tau - \xi \ar^s \int &\widehat{(\p_t\rho)_T}(\tau - \lambda) \widetilde{\psi}(\lambda, \xi) d\lambda \Big\|_{L^2_{\tau, \xi}(\Omega_2 \cap \{ 2|\tau - \xi| \les |\xi|\})}\\
            &\lesa T^{-1} \Big\| \al \tau + \xi \ar^{b-1 -s} \al \tau - \xi \ar^s \int  \widehat{(\p_t\rho)_T}(\tau - \lambda) \al \xi \ar^s \widetilde{\psi}(\lambda, \xi) d\lambda \Big\|_{L^2_{\tau, \xi}(\Omega_2)} \\
            &\lesa T^{s - b} \|\psi \|_{Y^{s, 0}_+} \sup_{\xi, \lambda} \big\| \al \tau - \xi \ar^s \widehat{(\p_t\rho)_T}(\tau - \lambda) \big\|_{L^2_\tau}.
    \end{align*}
To control the $\p_t \rho$ term we use
        \begin{align*}
            \big\| \al \tau - \xi \ar^s \widehat{(\p_t\rho)_T}(\tau - \lambda) \big\|_{L^2_\tau}
                        &\lesa \big\| \al \tau - \xi \ar^s \widehat{(\p_t\rho)_T}(\tau - \lambda) \big\|_{L^2_\tau(|\tau - \xi|\les T^{-1} )} + \big\| \al \tau - \xi \ar^s \widehat{(\p_t\rho)_T}(\tau - \lambda) \big\|_{L^2_\tau(|\tau - \xi | \g T^{-1})}\\
                        &\lesa \| \al \tau \ar^s \|_{L^2_\tau (|\tau|\les T^{-1})} \big\| \widehat{(\p_t\rho)_T}\big\|_{L^\infty} + T^{-s} \| \p_t\rho_T \|_{L^2} \\
                        &\lesa_\rho T^{\frac{1}{2} - s}
        \end{align*}
and so we can estimate the first term in (\ref{prop - energy for Z with time dependence - decomp into terms}).

Finally, to estimate the remaining term in (\ref{prop - energy for Z with time dependence - decomp into terms}), we write
     \begin{align*}
           \Big\| \al \tau + \xi \ar^{b-1} \al \tau - \xi \ar^s \int &\widehat{\rho_T}(\lambda - \tau ) \widetilde{F}(\lambda, \xi) d\lambda \Big\|_{L^2_{\tau, \xi}(\Omega_2)}\\
           &\lesa T^\epsilon \Big\| \al \tau + \xi \ar^{b-1+\epsilon} \al \tau - \xi \ar^s \int_{2|\tau + \xi| \les |\lambda + \xi|} \widehat{\rho_T}(\lambda - \tau ) \widetilde{F}(\lambda, \xi) d\lambda \Big\|_{L^2_{\tau, \xi}(\Omega_2)}\\& \qquad \qquad +
           T^\epsilon \Big\| \al \tau + \xi \ar^{b-1+\epsilon} \al \tau - \xi \ar^s \int_{2|\tau + \xi|\g |\lambda + \xi|} \widehat{\rho_T}(\lambda - \tau ) \widetilde{F}(\lambda, \xi) d\lambda \Big\|_{L^2_{\tau, \xi}(\Omega_2)}.
    \end{align*}
In the region $2|\tau + \xi | \les |\lambda + \xi|$ we have $|\lambda  + \xi | \approx | \tau - \lambda|$ and so, using the fact that $|\tau + \xi| \g T^{-1}$,
\begin{align*}
    \Big\| \al \tau + \xi \ar^{b-1+\epsilon} \al \tau - \xi \ar^s \int_{2|\tau + \xi| \les |\lambda + \xi|} &\widehat{\rho_T}(\lambda - \tau ) \widetilde{F}(\lambda, \xi) d\lambda \Big\|_{L^2_{\tau, \xi}(\Omega_2)}\\
            &\lesa \Big\| \al \tau - \xi \ar^s \int  (T|\tau - \lambda|)^{1-b-\epsilon} \big|\widehat{\rho_T}(\tau - \lambda )\big| \al \lambda + \xi \ar^{b-1+\epsilon} \big|\widetilde{F}(\lambda, \xi)\big| d\lambda \Big\|_{L^2_{\tau, \xi}}\\
            &\lesa_\rho \| F\|_{Z^{s, b-1+\epsilon}_+}
\end{align*}
where the last line follows from an application of Lemma \ref{lem - time dependence of Sobolev norms}. On the other hand, if $2|\tau + \xi | \g |\lambda + \xi|$, we can simply use the estimate $\al \tau + \xi \ar^{b-1+\epsilon} \lesa \al \lambda + \xi \ar^{b-1+\epsilon}$
followed by another application of Lemma \ref{lem - time dependence of Sobolev norms}. Therefore we have
        $$\Big\| \al \tau + \xi \ar^{b-1} \al \tau - \xi \ar^s \int \widehat{\rho_T}(\lambda - \tau ) \widetilde{F}(\lambda, \xi) d\lambda \Big\|_{L^2_{\tau, \xi}(\Omega_2)}\lesa T^\epsilon \| F\|_{Z^{s, b-1+\epsilon}_+}$$
and consequently the result follows.

\end{proof}

\end{proposition}

We remark that the factor $T^{\frac{1}{2} - b}$ in front of the term $\| \psi \|_{Y^{s, 0}_\pm}$ in (\ref{prop - energy for Z with time dependence - main eqn}) is not ideal as for $T$ small, this will blow up since $b>\frac{1}{2}$. This cannot be avoided, as a simple scaling argument shows that this is in fact the best possible exponent on $T$. Essentially the problem comes because the spaces $Y^{s, 0}$ and $Z^{s, b}$ scale differently, more precisely the $Y^{s, 0}$ space scales like $Z^{s, b}$ at the endpoint $b=\frac{1}{2}$. However, the term $T^{\frac{1}{2} - b}$ is not a huge problem, as if we can take $b$ sufficiently close to $\frac{1}{2}$, then we can safely absorb this into the inhomogeneous term $T^{\epsilon} \| F \|_{Y^{s, \epsilon -1}_\pm}$ in Lemma \ref{lem - Y energy est}.

\begin{corollary}\label{cor - energy inequality with time dilation}
Let $\frac{-1}{2}<s<0$, $0<\epsilon<\frac{1}{6}$, and $\frac{1}{2}<b < \min\{ 1  + \epsilon, 1+s\}$. Assume $0<T<1$ and define $S_T=[-T, T]\times \RR$. Let $\psi$ be the solution to
            $$ \p_t \psi \pm \p_x \psi = F$$
with $\psi(0) = f$. Then
        $$ \| \psi \|_{Z^{s, b}_\pm(S_T)} \lesa T^{\frac{1}{2} - b} \| f\|_{H^s} + T^{\epsilon} \inf_{F' = F \text{ on } S_T} \Big(\| F' \|_{Y^{s, -1 + 2\epsilon}_\pm} + \| F'\|_{Z^{s, b-1+2\epsilon}_\pm}\Big).$$
\begin{proof}
  Follows from Lemma \ref{lem - Y energy est} and Proposition \ref{prop - energy for Z with time dependence}.
\end{proof}
\end{corollary}

We now come to the proof of Proposition \ref{prop - uniqueness}.

\begin{proof}[Proof of Proposition \ref{prop - uniqueness}.]
It is enough to consider the case $\frac{-1}{2}<r \les s<0$. Choose $\epsilon>0$ sufficiently small such that
            \begin{equation}\label{prop - uniqueness - assump 1} r >\frac{-1}{2} + 4 \epsilon \end{equation}
and
             \begin{equation} \label{prop - uniqueness - assump 2} \frac{1}{2}<b< \frac{1}{2} + \epsilon.\end{equation}
A standard argument using Corollary \ref{cor - energy inequality with time dilation} reduces the problem to obtaining the estimates
    \begin{align}\label{prop - uniqueness - ineq Z}
      \| \psi \phi \|_{Z^{s, b - 1 + 2\epsilon}_\pm} &\lesa \| \psi \|_{Z^{s, b}_\pm} \| \phi \|_{Z^{r, b}_\mp},
    \\
       \| \psi \phi \|_{Y^{s, 2\epsilon-1}_\pm} &\lesa \| \psi \|_{Z^{s, b}_\pm} \| \phi \|_{Z^{r, b}_\mp}, \label{prop - uniqueness - ineq Y}\\
      \| \psi \|_{Y^{s, 2\epsilon -1 }_\pm} &\lesa \| \psi \|_{Z^{s, b}_\mp}. \label{prop - uniqueness - mass term}
    \end{align}
We start with (\ref{prop - uniqueness - ineq Z}). By Lemma \ref{lem - Z product estimate} we need
            $$ s \prec \{ s, b\}, \qquad b-1 + 2 \epsilon \prec \{ b, r\}.$$
The first condition is straight forward since $s>\frac{-1}{2}$ and $b>\frac{1}{2}$. For the second we need
            $$b + r \g0, \qquad b-1 + 2\epsilon \les \min\{ b, r\}, \qquad b-1+2\epsilon < b + r - \frac{1}{2}$$
which all hold in view of the assumptions (\ref{prop - uniqueness - assump 1}) and (\ref{prop - uniqueness - assump 2}).

To prove (\ref{prop - uniqueness - ineq Y}), we observe that by an application of the triangle inequality on the Fourier transform side, it suffices to show that
            $$ \| \psi \phi \|_{Y^{s, -1}_\pm} \lesa \| \psi \|_{Z^{s, b-2\epsilon}_\pm} \| \phi \|_{Z^{r - 2\epsilon, b}_\mp}.$$
By letting $a_0 = s$ and $b_0 = r- 4\epsilon$ in Lemma \ref{lem - Y nonlinear est}, we can reduce this to showing
        \begin{equation}\label{prop - uniqueness - Y est cond} \begin{split} s \prec \{s, b\}, \qquad r - 4\epsilon \prec \{ b-2\epsilon, r - 2\epsilon\},  \qquad s \prec\{ s , r +1 - 4 \epsilon \} \\
                            s + b - 2\epsilon> \frac{-1}{2}, \qquad r + b - 2\epsilon >  \frac{-1}{2}. \end{split} \end{equation}
The first condition is obvious. For the second condition we need
        $$ b - 2 \epsilon + r - 2\epsilon \g 0, \qquad r - 4\epsilon \les \min\{ b-2 \epsilon, r-2\epsilon\}, \qquad r-4\epsilon< b-2\epsilon + r-2\epsilon - \frac{1}{2}$$
which all follow from (\ref{prop - uniqueness - assump 1}) and (\ref{prop - uniqueness - assump 2}). The third condition in (\ref{prop - uniqueness - Y est cond}) can be written as
        $$ s + r + 1 - 4 \epsilon\g0, \qquad s \les \min\{ s, r+1  - 4 \epsilon\}, \qquad  s< s + r + 1 - 4\epsilon - \frac{1}{2}$$
and again each of these inequalities follows from (\ref{prop - uniqueness - assump 1}), (\ref{prop - uniqueness - assump 2}) and $r\les s <0$. The remaining conditions in (\ref{prop - uniqueness - Y est cond}) are also easily seen to be satisfied and so (\ref{prop - uniqueness - ineq Y}) follows.

Finally to prove (\ref{prop - uniqueness - mass term}) we use Holder's inequality to obtain
        \begin{align*}
           \| \psi \|_{Y^{ s,  2\epsilon - 1} _\pm} = \| \al \xi \ar^s \int_\RR \al \tau \pm \xi\ar^{2\epsilon-1} |\widehat{\psi}| d\tau \|_{L^2_\xi} &\lesa \| \al \xi \ar^s  \widehat{\psi}  \|_{L^2_{\tau, \xi}} \\
           &\lesa \| \al \tau \mp \xi \ar^s \al \tau \pm \xi \ar^{|s|} \widehat{\psi } \|_{L^2_{\tau, \xi}}\\
           &\les \|\psi \|_{Z^{s, b}_\mp}.
        \end{align*}
\end{proof}

\section{Global Existence}\label{sec global existence}

Here we prove Corollary \ref{cor - global existence}.

\begin{proof}[Proof of Corollary \ref{cor - global existence}]
The persistence of regularity in Theorem \ref{thm - local existence} shows that it suffices to prove global existence in the case $s=0$ and $\frac{-1}{2}< r\les 0$. Let $(u_\pm, A_\pm)$ be the solution to (\ref{Dirac pm}) and (\ref{Gauge pm}) given by Theorem \ref{thm - local existence}   with initial data $(f_\pm, a_\pm) \in L^2 \times H^r$. We extend $(u_\pm, A_\pm)$ to some maximal interval of existence $(-T, T)$. To show the solution is global in time, it is enough to show that if $T<\infty$ then we have the bound
            \begin{equation}\label{cor - global existence - bound on A} \sup_{t \in (-T, T)} \|A_\pm(t) \|_{H^r} <\infty. \end{equation}
Since supposing (\ref{cor - global existence - bound on A}) holds, we can extend the solution past $(-T, T)$ by using the $L^2$ conservation of $u_\pm$, together with the local well-posedness of Theorem \ref{thm - local existence}. Thus contradicting the fact that $(-T, T)$ was the maximal time of existence. Consequently we must have $T=\infty$.

 To obtain the bound (\ref{cor - global existence - bound on A}) we make use of the following decomposition first used in \cite{Candy2010} based on an idea due to Delgado \cite{Delgado1978}. We split the Dirac component of our solution $u_\pm$ into a mass free part $u^L_\pm$ satisfying
        \begin{align*}  i ( \p_t u^L_\pm \pm \p_x u^L_\pm ) &= - A_{\mp} u^L_\pm   \\
                                                       u^L_\pm(0) &= f_\pm \end{align*}
and a term $u_\pm^N$ with vanishing initial data
         \begin{align*}  i ( \p_t u^N_\pm \pm \p_x u^N_\pm ) &= m u_\mp- A_{\mp} u^N_\pm   \\
                                                       u^N_\pm(0) &= 0. \end{align*}
Observe that $u_\pm = u^L_\pm + u^N_\pm$. Since $A_\pm$ is real valued, a computation shows that
        $$ \p_t \big| u^L_\pm \big|^2 \pm \p_x \big|u_\pm^L\big|^2 = 0$$
and
       $$ \p_t \big|u^N_\pm\big|^2 \pm \p_x \big|u^N_\pm\big|^2 = 2m \Im( u_\mp \overline{u}^N_\pm). $$
Hence
        \begin{equation}\label{cor - global existence - L est}
                | u_\pm^L(t, x) | = | f_\pm (x \mp t) |
        \end{equation}
and, via the Duhamel formula\footnote{For more detail see Proposition 7 in \cite{Candy2010}.},
        \begin{equation}\label{cor - global existence - N est}
            \sup_{ |t|< T} \big( \| u^N_+(t) \|_{L^\infty_x} + \| u^N_-(t) \|_{L^\infty_x} \big) \lesa_{T, m} \| f_+\|_{L^2} + \| f_- \|_{L^2}.
        \end{equation}

To obtain the bound (\ref{cor - global existence - bound on A}),  we note that the equation for $A_\pm$ easily leads to
        \begin{equation}
            \| A_+(t)\|_{H^r_x} + \| A_-(t) \|_{H^r_x} \lesa \| a_+ \|_{H^r} + \| a_- \|_{H^r} + \int_0^t \| u_+ \overline{u}_- \|_{L^2_x} ds
        \end{equation}
and so it suffices to bound $ \int_{|s|<T} \| u_+(s) \overline{u}_-(s)  \|_{ L^2_x} ds$ in terms of the initial data $f_\pm$. If we now use the decomposition $u_\pm = u^L_\pm + u^N_\pm$ we have
    \begin{align}
         u_+ \overline{u}_- =  u_+ \overline{u}_-^L + u_+ \overline{u}_-^N = u_+^L \overline{u}_-^L + u_+^N \overline{u}_-^L + u_+ \overline{u}_-^N .
    \end{align}
The terms involving $u_\pm^N$ are straightforward by (\ref{cor - global existence - N est}), while for the remaining term Holder's inequality followed by a change of variables gives
                $$    \int_{|s|<T} \| u_+^L(s) \overline{u}_-^L(s) \|_{L^2_x(\RR)} ds \lesa_T \| f_+(x-s) \overline{f}_-(x+s) \|_{L^2_{s, x}(\RR^2)} \lesa \|f_+\|_{L^2} \| f_-\|_{L^2}.$$
Therefore the required bound (\ref{cor - global existence - bound on A}) follows.
\end{proof}

%lem - time dependence of Sobolev norms - main eqn 1

\section*{Appendix - Proof of  \eqref{lem - time dependence of Sobolev norms - main eqn 1}}
Here we will sketch the proof of \eqref{lem - time dependence of Sobolev norms - main eqn 1}. This result is essentially well-known, but for the readers convenience we will give the outline of the proof.
\begin{proof}[Proof of estimate \ref{lem - time dependence of Sobolev norms - main eqn 1}]
  We start by noting that the inequality (\ref{lem - time dependence of Sobolev norms - main eqn 1}) follows immediately from the estimates
           \begin{equation}\label{appendix - besov product estimate} \| f g\|_{H^s} \lesa \|f \|_{B^{\frac{1}{2}}_{2, 1}} \| g\|_{H^s}\end{equation}
  and
            \begin{equation}\label{appendix - besov scale invariance} \| \rho_T(t) \|_{B^{\frac{1}{2}}_{2, 1}} \lesa  \| \rho \|_{B^{\frac{1}{2}}_{2, 1}}\end{equation}
  where $ \| f\|_{B^{\frac{1}{2}}_{2, 1}} = \sum_{N \in 2^{\NN}} N^{\frac{1}{2}} \|  f_{ N} \|_{L^2}$ and
    $$f_{ N} =\widehat{P_{ N} f} = \chi_{\{|\xi| \sim N\}} \widehat{f}$$
    for $N>1$ with $\widehat{ f}_{ 1} = \chi_{\{|\xi|\lesa 1\}} \widehat{f}$. We use $\chi_\Omega$ to denote the characteristic function of the set $\Omega$. We also use the notation $\widehat{f}_{\ll N} = \chi_{\{ |\xi|\ll N\}}\widehat{f}$. To prove (\ref{appendix - besov product estimate}) we recall the characterisation
            $$ \| f \|_{H^s}^2 \approx \sum_{N \in 2^{\NN}} N^{2s} \| f_N \|_{L^2}^2.$$
  as well as the Trichotomy formula
            $$ P_{N}( fg) \approx f_{\ll N} g_N + f_N g_{\ll N} + \sum_{M \g N} P_{N}( f_M g_M)$$
  where the sum is over dyadic numbers $M \in 2^{\NN}$. We estimate each of these terms separately. For the first term we observe that
    $$  \| f_{\ll N} g_N \|_{L^2} \lesa \| \widehat{f}_{\ll N} \|_{L^1} \| \widehat{g}_N \|_{L^2} \lesa \Big(\sum_{M\ll N} M^{\frac{1}{2}} \| f_M \|_{L^2} \Big) \| g_N \|_{L^2} $$
  and so
    \begin{align*}
        \sum_{N \in 2^{\NN}} N^{2s} \| f_{\ll N} g_N \|_{L^2}^2 &\lesa  \sum_{N \in 2^{\NN}} N^{2s}  \Big(\sum_{M\ll N} M^{\frac{1}{2}} \| f_M \|_{L^2} \Big)^2 \| g_N \|_{L^2}^2\\
                    &\lesa  \Big(\sum_{M \in 2^{\NN}} M^{\frac{1}{2}} \| f_M \|_{L^2} \Big)^2\sum_{N \in 2^{\NN}} N^{2s} \| g_N \|_{L^2}^2
                    \approx \| f\|_{B^{\frac{1}{2}}_{2, 1}}^2 \| g\|_{H^s}^2.
        \end{align*}
To estimate the term $f_N g_{\ll N}$ a similar computation gives
        \begin{align*}
            \sum_{N \in 2^{\NN}} N^{2s} \| f_{ N} g_{\ll N} \|_{L^2}^2 &\lesa  \sum_{N \in 2^{\NN}} N^{2s}  \Big(\sum_{M\ll N} M^{\frac{1}{2}} \| g_M \|_{L^2} \Big)^2 \| f_N \|_{L^2}^2. \end{align*}
Now since $s<\frac{1}{2}$, we have
        $$ \Big( \sum_{M \ll N} M^{\frac{1}{2}} \| g_M\|_{L^2}\Big)^2 \lesa \Big( \sum_{M\ll N} M^{1-2s} \Big) \Big( \sum_{M \ll N } M^{2s} \| g_M \|_{L^2}^2\Big) \lesa N^{1-2s} \| g\|_{H^s}$$
and therefore
        $$ \sum_{N \in 2^{\NN}} N^{2s} \| f_{ N} g_{\ll N} \|_{L^2}^2 \lesa \| g\|_{H^s}^2 \sum_{N \in 2^{\NN} } N \| f_N\|_{L^2}^2 \approx \| f\|_{H^{\frac{1}{2}}}^2 \| g\|_{H^s}^2 \lesa \| f\|_{B^{\frac{1}{2}}_{2, 1}}^2 \| g\|_{H^s}^2.$$
Finally, for the remaining term $\sum_{M>N} P_N( f_M g_M)$, we note that
        \begin{align*}
            \big\| \sum_{M>N} P_N(f_M g_M) \big\|_{L^2} &\lesa \sum_{M>N} \| P_N(f_M g_M) \|_{L^2} \\
                                            &\lesa \sum_{M>N} N^{\frac{1}{2}} \| f_M\|_{L^2} \| g_M \|_{L^2} \\
                                            &\lesa N^{\frac{1}{2}} \Big( \sum_{M>N} M^{-2s} \|f_M\|_{L^2}^2\Big)^\frac{1}{2} \| g\|_{H^s}. \end{align*}
Hence, for $s>\frac{-1}{2}$,
    \begin{align*}
        \sum_{N \in 2^{\NN}} N^{2s} \big\| \sum_{M>N} P_N( f_{ M} g_{M}) \big\|_{L^2}^2 &\lesa \| g\|_{H^s}^2 \sum_{N \in 2^{\NN}} N^{2s+1} \sum_{M>N} M^{-2s} \|f_M\|_{L^2}^2 \\
        &\lesa \| g\|_{H^s}^2 \sum_{M \in 2^{\NN}} M^{-2s} \| f_M\|_{L^2}^2\sum_{N<M} N^{1+2s} \\
        &\lesa \| g\|_{H^s}^2 \sum_{M \in 2^{\NN}} M \| f_M\|_{L^2}^2 \\
        &\lesa \| g\|_{H^s}^2 \| f\|_{B^{\frac{1}{2}}_{2, 1}}^2
    \end{align*}
and so (\ref{appendix - besov product estimate}) follows.

The inequality (\ref{appendix - besov scale invariance}) follows by using the characterisation\footnote{See for instance Theorem 7.47 in \cite{Adams2003}.}
            $$ \| f\|_{B^{\frac{1}{2}}_{2, 1}} \approx \| f\|_{L^2} + \int_{\RR} \big\| f(x) - f(x - y) \big\|_{L^2_x} \frac{dy}{|y|^{1+\frac{1}{2}}} $$
together with a change of variables.
\end{proof}

\bibliographystyle{amsplain}
%\bibliographystyle{amsalpha}
%\bibliography{CSDpaper}
\providecommand{\bysame}{\leavevmode\hbox to3em{\hrulefill}\thinspace}
\providecommand{\MR}{\relax\ifhmode\unskip\space\fi MR }
\providecommand{\MRhref}[2]{%
  \href{http://www.ams.org/mathscinet-getitem?mr=#1}{#2}
}
\providecommand{\href}[2]{#2}

\end{document}